\documentclass[11pt]{article}
\usepackage[margin=2.5cm]{geometry}
\usepackage{graphicx} 
\usepackage{todonotes}
\usepackage{xcolor}
\usepackage{tikz}
\usepackage{amsfonts,amsmath,amssymb,amsthm,mathtools,mathrsfs,dsfont,bm,mathabx}

\usepackage{chngcntr}
\usepackage{enumitem}
\usepackage{microtype}
\usepackage{color}
\usepackage{float}
\usepackage{array}
\usepackage{multirow}
\usepackage{hyperref}
\usepackage{caption}
\usepackage{booktabs}
\usepackage{subcaption}
\usepackage{cleveref}
\usepackage{algorithm}
\usepackage{algpseudocode}

\crefname{prop}{proposition}{propositions}
\Crefname{prop}{Proposition}{Propositions}

\usetikzlibrary{calc,fadings,positioning,decorations.pathmorphing,intersections,patterns,decorations.markings,svg.path,graphs,matrix,arrows.meta,arrows, mindmap, graphs, graphs.standard,shapes.geometric,calligraphy,decorations.pathreplacing}

\DeclareMathOperator{\dom}{DOM}
\DeclareMathOperator{\tdom}{TDOM}

\DeclareMathOperator{\forb}{Forb}

\newcommand{\cF}{\mathcal{F}}

\newcommand{\wN}{\widetilde{N}}

\newcommand{\cS}{\mathcal{S}}

\newcommand{\cE}{\mathcal{E}}

\newcommand{\cP}{\mathcal{P}}
\newcommand{\cQ}{\mathcal{Q}}
\newcommand{\cH}{\mathcal{H}}
\newcommand{\sG}{\mathscr{G}}
\newcommand{\cstN}{\wN_{p,K}}

\newcommand{\UF}{\bigcup \cF}

\newcommand{\card}[1]{\left\lvert #1\right\rvert}
\newcommand{\pr}[1]{\mathbb{P}\left[#1\right]}

\newcommand{\sst}[2]{\left\{#1 : #2\right\}}
\newcommand{\pth}[1]{\left(#1\right)}
\newcommand{\floor}[1]{\left\lfloor#1\right\rfloor}
\newcommand{\ceil}[1]{\left\lceil#1\right\rceil}

\newcommand{\bsig}{\bm{\sigma}}

\newtheorem{thm}{Theorem}
\newtheorem{prop}[thm]{Proposition}
\newtheorem{cor}[thm]{Corollary}

\newtheorem{conj}[thm]{Conjecture}
\newtheorem{lemma}[thm]{Lemma}
\newtheorem{claim}{Claim}[thm]
\newtheorem{problem}{Problem}
\newtheorem{obs}{Observation}
\theoremstyle{definition}

\newtheoremstyle{case}{}{}{}{}{}{:}{ }{}
\theoremstyle{case}

\counterwithin*{case}{claim}
\theoremstyle{remark}
\newtheorem{rk}{Remark}

\newenvironment{proofclaim}[1][{\it Proof of claim. \hspace{0.066cm}}]%
	{\noindent {}{#1}{}}{ \strut\hfill $\lozenge$\vspace{2ex}}

\title{A Domatic Analogue of $\chi$-Bounded Graph Classes and the Gy\' arf\' as--Sumner Conjecture}
\author{Quentin Chuet, Selma Djelloul, Hoang La, François Pirot, Hossein Zaredehabadi}

\begin{document}

\maketitle

\begin{abstract}
Given a graph $G$, a dominating set is a subset $X\subseteq V(G)$ such that $N[X]=V(G)$. The \emph{domatic number} of $G$, denoted $\dom(G)$, is the maximum size of a partition of $V(G)$ into dominating sets. 
In analogy with the lower bound of the chromatic number by the clique number, the domatic number satisfies the upper bound
$\dom(G)\le \delta(G)+1$ where $\delta(G)$ is the minimum degree of $G$.
Therefore,
as an analogue of the notion of $\chi$-bounded graph classes, we say that a class of graphs $\sG$ is \emph{DOM-bounded} if there exists a positive unbounded function $f_\sG$ such that for every $G\in \sG$, 
we have $\dom(G) \ge f_\sG(\delta(G))$.

We propose the following conjecture for graphs forbidding a fixed induced subgraph, analogous to the Gy\'arf\'as--Sumner Conjecture for $\chi$-bounded graph classes: for every connected graph $H$, the class of $H$-free graphs is DOM-bounded if and only if $H$ is a tree of diameter at most $3$.
We reduce the case of disconnected graphs to the connected setting and show that the conditions on $H$ are necessary.

We show that star-free graphs of minimum degree at least $\delta$ have domatic number $\Omega(\delta /\log \delta)$, which is best possible up to a constant factor. 
We also identify a subclass of star-free graphs for which the domatic number is linear in $\delta$: line graphs of bounded rank hypergraphs. 

In support of our conjecture in the case of double stars, we prove that $P_4$-free graphs (i.e. cographs) of minimum degree $\delta$ have domatic number at least $1 + \frac{\delta}{2}$, which is best possible. 
\end{abstract}

\section{Introduction}
Let $G$ be a graph and $k$ an integer. Given a $k$-colouring $\phi\colon V(G) \to [k]$, we say that $\phi$ is \emph{dominating} if $\phi(N[v])=[k]$ for every vertex $v\in V(G)$. Equivalently, the colour classes of $\phi$ form a \emph{domatic partition} of $G$, that is, a partition of $V(G)$ into dominating sets (recall that a subset $X\subseteq V(G)$ is \emph{dominating} if $N[X]=V(G)$). The \emph{domatic number $\dom(G)$ of $G$} is the maximum $k$ such that a dominating $k$-colouring of $G$ exists.

A standard upper bound is $\dom(G)\le \delta(G)+1$, where $\delta(G)$ denotes the minimum degree of $G$. If equality holds, then $G$ is said to be \emph{domatically full}. The largest known class of domatically full graphs is that of strongly chordal graphs \cite{Far84}.
The equality between the chromatic number and the clique in perfect graphs, and more generally, their functional equivalence in $\chi$-bounded classes naturally parallels the relationship between the domatic number and the minimum degree in domatically full graphs, and more broadly, in DOM-bounded classes which we define as follows.
Given a class of graphs $\sG$, we say that $\sG$ is \emph{DOM-bounded} if there exists a positive unbounded function $f_\sG$ such that
\[
\dom(G) \ge f_\sG(\delta(G))
\]
for every graph $G\in \sG$.
We say that $f_\sG$ is a \emph{DOM-binding function of $\sG$}, and if there is no graph $G\in \sG$ such that $\dom(G) > f_\sG(\delta(G))$, we further say that $f_\sG$ is \emph{optimal}. 
This condition is trivially satisfied when $\sG$ has bounded minimum degree; consequently, we restrict our attention to classes with arbitrarily large minimum degree.
In particular, non-trivial minor-closed classes of graphs are not relevant with respect to DOM-boundedness, as every $H$-minor-free graph (for a fixed $t$-vertex graph $H$) is $O(t\sqrt{\log t})$-degenerate \cite{kostochka1982minimum}. 

Zelinka \cite{zelinka1983domatic} proved that the class of all graphs is not DOM-bounded by constructing graphs with domatic number $2$ and arbitrarily large minimum degree (see \Cref{sec:constructions}). By modifying Zelinka's construction (which is bipartite), we obtain split graphs and high-girth graphs with the same property.
Hence, bipartite graphs, split graphs, and graphs of high girth are not DOM-bounded. The objective of this paper is to identify structural constraints that guarantee DOM-boundedness.


We turn our attention to the classes of graphs with forbidden sub-structures. 

\paragraph{Forbidden subgraphs}
Consider the class $\sG$ of graphs that do not contain a fixed graph $H$ as a subgraph. If $H$ contains a cycle of length $\ell$, then $\sG$ contains the class of graphs of girth greater than $\ell$, which is not DOM-bounded. Otherwise, $H$ is a forest, and every graph in $\sG$ is $(|V(H)|-2)$-degenerate, and therefore has bounded minimum degree. Thus, forbidding a subgraph does not provide a meaningful structural constraint with respect to DOM-boundedness.

\paragraph{Forbidden induced subgraphs}
A graph $G$ is \emph{$H$-free} if it does not contain $H$ as an induced subgraph; we denote by $\forb(H)$ the class of $H$-free graphs. More generally, given a family $\cF$ of graphs, we say that $G$ is $\cF$-free if it is $H$-free for every $H\in \cF$, and we denote by $\forb(\cF)$ the corresponding class. Our goal is to characterise the graphs $H$ for which $\forb(H)$ is DOM-bounded.

The analogous problem for $\chi$-boundedness is the Gy\'arf\'as--Sumner Conjecture.

\begin{conj}[Gy\'arf\'as, 1975~\cite{Gya75}; Sumner, 1981~\cite{Sum81}]
\label{conj:gyarfas-sumner}
    For every graph $H$, $\forb(H)$ is $\chi$-bounded if and only if $H$ is a forest.
\end{conj}

As in the case of forbidding $H$ as a (non-induced) subgraph, if $H$ contains a cycle of length $\ell$, then $\forb(H)$ contains the class of graphs of girth greater than $\ell$, which is not DOM-bounded.
If $H$ is connected and has diameter at least $4$, then it contains an induced copy of $P_5$.
Hence, $\forb(H)$ contains the class of split graphs (which are $P_5$-free) and is not DOM-bounded.
If $H$ has at least two connected components of size at least $2$, then $\forb(H)$ contains $\forb(2K_2)$, which in turn contains the class of split graphs. Hence, $\forb(H)$ is not DOM-bounded. 
Consequently, when $H$ is disconnected, $\forb(H)$ can be DOM-bounded only if there exist a connected graph $H_0$ and a positive integer $k$ such that $H$ is the disjoint union of $H_0$ and $k$ isolated vertices, and $\forb(H_0)$ is DOM-bounded (since $\forb(H_0) \subseteq \forb(H)$). We prove that the converse also holds.

\begin{thm} \label{thm:disconnected}
    Let $H$ be a disconnected graph. Then $\forb(H)$ is DOM-bounded if and only if $H$ has at most one connected component $H_0$ of order at least $2$, and $\forb(H_0)$ is DOM-bounded.
\end{thm}

The remaining graphs to consider are trees of diameter at most $3$. 
We conjecture that all such graphs yield DOM-bounded classes.

\begin{conj}
\label{conj:main}
    For every connected graph $H$, $\forb(H)$ is DOM-bounded if and only if $H$ is a tree of diameter at most $3$. 
\end{conj}

A tree of diameter at most $3$ is either a star (diameter $2$) or a double star (diameter $3$). We partially resolve \Cref{conj:main} by establishing the result when $H$ is a star.

\begin{thm}
    \label{thm:main}
    Let $k$ be an integer such that $k\geq 2$.
    If $G$ is a $K_{1,k+1}$-free graph of minimum degree $\delta \ge 4^{1+k^2}$, then
    \[
    \dom(G) \ge \floor{\frac{\delta}{3\log \delta}}.
    \]
\end{thm}

Note that $K_{1,1}$-free are edgeless graphs, and $K_{1,2}$-free graphs are disjoint unions of cliques, which are domatically full.
Furthermore, \Cref{thm:main} is tight up to an absolute constant factor, as shown in \Cref{prop:alpha2}.

We then investigate which subclasses of star-free graphs admit a linear DOM-binding function.
We prove it for line graphs of rank-$k$ hypergraphs (that is, hypergraphs with edge size at most $k$), which form a subclass of $\forb(K_{1,k+1})$.

\begin{thm}
    \label{thm:linegraph}
    Let $k$ be an integer such that $k\ge 2$.
    If $G$ is the line graph of a rank-$k$ hypergraph, then
    \[
    \dom(G) \ge \frac{\delta(G)}{k(\log k + O(\log \log k))}.
    \]
\end{thm}

Finally, we consider double-star-free graphs in support of \Cref{conj:main}. The smallest double star is $P_4$, and $\forb(P_4)$ coincides with the class of cographs \cite{Sei74}. We show that cographs are linearly DOM-bounded and determine the optimal DOM-binding function.

\begin{thm}
\label{thm:cographs}
For every cograph $G$,
\[
\dom(G) \ge 1+ \frac{\delta(G)}{2}.
\]
\end{thm}

The bound in \Cref{thm:cographs} is tight due to complete multipartite graphs where every part has size 2.

\paragraph{Outline of the paper}

In \Cref{sec:proba}, we introduce the probabilistic tools required in our proofs, namely a special case of
the Lovász Local Lemma. 
In \Cref{sec:constructions}, we present our constructions that illustrate the tightness
of our characterisation of DOM-bounded classes of graphs. 
In \Cref{sec:disconnected}, we reduce disconnected forbidden induced subgraphs to connected ones for DOM-boundedness. Namely, we show that if $H$ is a connected graph such that the class of $H$-free graphs is DOM-bounded, then so is the class of $(H + I)$-free graphs, where $I$ is any independent set. 
In \Cref{sec:linegraphs}, we prove that line graphs are DOM-bounded with a linear binding function.
In \Cref{sec:cographs}, we
prove that cographs are DOM-bounded with a linear binding function.
In \Cref{sec:stars}, we show
how to obtain a quasilinear lower bound for the domatic number in terms of the minimum degree with a randomised colouring. We first apply this strategy to unit disk graphs, then extend it to a more abstract setting that will be applicable to induced-$K_{1,t}$-free graphs.
Finally, in \Cref{sec:conclusion}, we present some open problems. 

\section{Preliminaries}
\label{sec:proba}

\subsection{Notation}

All graphs in this paper are finite and simple. 
For two vertex-disjoint graphs $G$ and $H$, we denote by $G+H$ their disjoint union. 
If $r$ is a nonnegative integer, we write $G+rK_1$ for the graph obtained from $G$ by adding $r$ isolated vertices. 
If $X\subseteq V(G)$, we write $G\setminus X$ for the induced subgraph $G[V(G)\setminus X]$; if $v\in V(G)$, we write $G\setminus v$ instead of $G\setminus\{v\}$.

\subsection{Total domatic number}

Given a graph $G$, a \emph{total dominating set} of $G$ is a subset $X\subseteq V(G)$ such that every vertex $v\in V(G)$ has at least one neighbour in $X$. Note that no total dominating set exists in $G$ if $G$ has an isolated vertex; otherwise $V(G)$ is a total dominating set of $G$. Observe also that all total dominating sets of a loopless graph $G$ have size at least $2$. 
We denote $\tdom(G)$ the \emph{total domatic number} of $G$, that is the maximum size of a partition of $V(G)$ into total dominating sets, which we call a \emph{total domatic partition of $G$}. This corresponds to the colour classes of a \emph{coupon colouring of $G$}, as defined in \cite{chen2015coupon}, which in turn corresponds to a panchromatic colouring \cite{kos02} (sometimes also called polychromatic colouring, e.g. in \cite{bollobas2013cover}) of its (open) neighbourhood hypergraph.

\begin{lemma}\label{lem:dom-tdom}
    For every graph $G$,
    \[\floor{\frac{\dom(G)}{2}} \le \tdom(G) \le \dom(G).\]
\end{lemma}

\begin{proof}
    Let $G$ be a graph.
    The bound $\tdom(G) \le \dom(G)$ is trivial since a total dominating set of $G$ is always, in particular, a dominating set of $G$. Let us prove the other inequality.
    The union of two disjoint dominating sets of $G$ is a total dominating set. So, given a domatic partition $\cP$ of $G$ of size $k$, one obtains a total domatic partition of $G$ of size $\floor{k/2}$ by merging the first $2\floor{k/2}-2$ parts of $\cP$ by pairs, and the last two or three ones together.
\end{proof}

Observe that the first inequality is best possible, as witnessed by $G=K_n$, which has $\dom(G)=n$ and $\tdom(G)=\floor{n/2}$.
The second inequality is also best possible, as witnessed by $G=K_{n,n}$, which has $\dom(G)=\tdom(G)=n$.
One could define the notion of TDOM-boundedness similarly to that of DOM-boundedness; \Cref{lem:dom-tdom} implies that these two notions are equivalent.

\subsection{Probabilistic tools}

We rely on the two following corollaries of the general Lov\'asz Local Lemma, of increasing strength and technicality. 

\begin{thm}[LLL]
    \label{thm:LLL}
    Let $A_1,...,A_n$ be events in a probability space $\Omega$ with dependency graph $\Gamma$.  If there exist real numbers $0 \le x_1,...,x_n < 1$ such that, for every $i\in [n]$,
    \[\pr{A_i} \le x_i \prod_{A_j \in N_\Gamma(A_i)} (1-x_j),\]
    then 
    \[\pr{\bigwedge_{i \in [n]} \overline{A_i}} \ge \prod_{i\in [n]} (1-x_i) > 0.\]
\end{thm}

The following corollary appears almost verbatim in \cite{MoRe02}, albeit with a slightly stronger hypothesis. We include its short proof for completeness.

\begin{cor}
    \label{cor:LLL}
    Let $A_1,...,A_n$ be events in a probability space $\Omega$ with dependency graph $\Gamma$. If, for all $i\in [n]$, one has $\pr{A_i}<1$ and 
    \begin{equation}
        \label{eq:LLL}
        \sum_{A_j \in N_\Gamma(A_i)} \pr{A_j} \le \frac{1}{4},
    \end{equation} 
    then $\pr{\bigwedge_{i \in [n]} \overline{A_i}} >0$.
\end{cor}

\begin{proof}
    If there is an isolated vertex $i$ in $\Gamma$, then $A_i$ is independent from the other events, so 
    \[\pr{\bigwedge_{j\neq i} \overline{A_j}} >0 \implies \pr{\bigwedge_{j \in [n]} \overline{A_j}} >0.\]
    We therefore assume that $\Gamma$ contains no isolated vertex. In particular, we have $\pr{A_i}\le 1/4$ for all $i\in [n]$. 
    We let $x_i \coloneqq 2\pr{A_i} \le 1/2$ for each $i\in [n]$.
    Using the inequality $1-x \ge 4^{-x}$ for all $0\le x \le 1/2$, we obtain
    \begin{align*}
        x_i \prod_{ij \in E(\Gamma)} (1-x_j) &\ge x_i 4^{-\sum_{ij \in E(\Gamma)} x_j} = x_i 4^{-2\sum_{ij \in E(\Gamma)} \pr{A_j}} \\
        & \ge x_i 4^{-2 \times \frac{1}{4}} = \frac{x_i}{2} = \pr{A_i},
    \end{align*}
    for all $i\in [n]$. The conclusion follows from \Cref{thm:LLL}.
\end{proof}

The following more technical corollary lets us introduce a trade-off between the cumulative probabilities of neighbouring events of a given random event, and its own probability.

\begin{cor}
    \label{cor:LLL2}
    Let $A_1,...,A_n$ be events in a probability space $\Omega$ with dependency graph $\Gamma$. If there exists $\alpha < 1$ such that, for all $i\in [n]$, one has $\pr{A_i}^\alpha\le \frac{1}{2}$ and 
    \begin{equation}
        \label{eq:LLL2}
        \sum_{A_j \in N_\Gamma(A_i)} \pr{A_j}^\alpha \le (1-\alpha)\log_4 \frac{1}{\pr{A_i}},
    \end{equation} 
    then $\pr{\bigwedge_{i \in [n]} \overline{A_i}} >0$.
\end{cor}

\begin{proof}
    We let $x_i \coloneqq \pr{A_i}^\alpha \le 1/2$ for each $i\in [n]$.
    Using the inequality $1-x \ge 4^{-x}$ for all $0\le x \le 1/2$, we obtain
    \begin{align*}
        x_i \prod_{ij \in E(\Gamma)} (1-x_j) &\ge x_i 4^{-\sum_{ij \in E(\Gamma)} x_j} = x_i 4^{-\sum_{ij \in E(\Gamma)} \pr{A_j}^\alpha} \\
        & \ge x_i 4^{(\alpha-1)\log_4 \frac{1}{\pr{A_i}}} =  \pr{A_i},
    \end{align*}
    for all $i\in [n]$. The conclusion follows from \Cref{thm:LLL}.
\end{proof}

In order to analyse a random construction in \Cref{sec:constructions}, we shall also use a Chernoff bound.

\begin{thm}\label{thm:Chernoff}
    Let $X_1,\dots,X_n$ be independent random variables taking values in $\{0,1\}$, let $X \coloneqq X_1 + \dots + X_n$, and $\mu \coloneqq \mathbb{E}(X)$. Let $0 \le \delta \le 1$. Then
    \[\pr{X \le (1-\delta)\mu} \le e^{-\delta^2\mu/2}.\]
\end{thm}

\section{Constructions with small domatic number and high minimum degree}
\label{sec:constructions}

Zelinka \cite{zelinka1983domatic} showed that there exist graphs of arbitrarily large minimum degree with bounded (total) domatic number. We repeat the short proof as a warm-up for the next constructions presented in this section, which adapt Zelinka's construction to satisfy some structural constraints. 

\begin{thm}[Zelinka, 1983]\label{thm:degree-tdom=1}
    For every integer $\delta \ge 2$, there exists a graph of minimum degree $\delta$ and total domatic number $1$.
\end{thm}
\begin{proof}
    Let $n \coloneqq 2d-1$, and let $G_d$ be the graph with vertex sets $A \cup B$, where $A \coloneqq [n]$ and $B \coloneqq \binom{A}{d}$, and edge set $\{ab \in A \times B : a \in b\}$. In any $2$-colouring $\phi$ of $G_d$, there exists a monochromatic subset $S \subseteq A$ of size $d$, which corresponds to an open neighbourhood in $G_d$, therefore $\phi$ is not a total dominating colouring. We conclude that $\tdom(G_d)=1$.
\end{proof}
\begin{rk}
    With a similar argument and $n \ge 3d - 2$, we have a graph of minimum degree $d$ and domatic number $2$.
    We remark that the structure of $G_d[A]$ does not impact the argument, thus adding any subset of edges to $A$ still yields a graph of minimum degree $d$ and total domatic number $1$. In particular, adding all possible edges within $A$ results in a split graph ($A$ is the clique, and $B$ is the independent set).
\end{rk}

Given a hypergraph $\cH$, the \emph{bipartite incidence graph} of $\cH$ is a graph with vertex set $V(\cH) \cup E(\cH)$, and such that $v \in V(\cH)$ is adjacent to $e\in E(\cH)$ if and only if $v \in e$. Observe that the graph $G_d$ constructed in the proof of \Cref{thm:degree-tdom=1} is precisely the bipartite incidence graph of $K_n^{(d)}$, the complete $d$-uniform hypergraph on $n \coloneqq 2d-1$ vertices.

The chromatic number of a hypergraph $\cH$ is the minimum number of colours needed to colour $V(\cH)$ in such a way that no edge of $\cH$ is monochromatic. The proof of \Cref{thm:degree-tdom=1} relied on the simple fact that $K_n^{(d)}$ has chromatic number at least $3$ if $n\ge 2d-1$, thus attempting to colour the vertices of $K_n^{(d)}$ with $2$ colours produces a monochromatic edge. Zelinka's construction can be generalised by considering the bipartite incidence graph of any $d$-uniform hypergraphs of \emph{chromatic number} at least $3$. In particular, if said hypergraph has no short ``cycles'', then its bipartite incidence graph has large girth.

A \emph{Berge cycle} in a hypergraph is a sequence $(v_1,e_1,v_2,e_2,\dots,v_\ell,e_\ell)$ of alternating and distinct vertices and edges, such that $\{v_i,v_{i+1}\} \subseteq e_i$ for $i \in [\ell-1]$ and $\{v_1,v_\ell\} \subseteq e_\ell$. The \emph{girth} of a hypergraph $\cH$, denoted $g(\cH)$, is the length of its smallest Berge cycle.
The existence of $d$-uniform hypergraphs with arbitrarily large girth and chromatic number was first proved by Erd\H{o}s \cite{erdos1959graph}. We refer to \cite{alon2016coloring} for references and a simple construction.

\begin{thm}[Erd\H{o}s, 1959]\label{thm:hypergraph-girth-chromatic}
    For every $k,g,d \ge 2$, there exists a $d$-uniform hypergraph $\cH$ of girth at least $g$, and chromatic number at least $k$.
\end{thm}

A hypergraph $\cH$ is called \emph{vertex-critical} if every proper induced subgraph of $\cH$ has smaller chromatic number. The degree of a vertex is the number of incident edges.

\begin{lemma}\label{lem:critical-hypergraph}
    Let $\cH$ be a vertex-critical hypergraph of chromatic number $k$. Then $\cH$ has minimum degree at least $k-1$.
\end{lemma}
\begin{proof}
    Suppose that $\cH$ has a vertex $v$ of degree less than $k-1$. By vertex-criticality, we have $\chi(\cH - v) < k$; let $\phi$ be a proper $(k-1)$-colouring of $\cH-v$. Adding back $v$ and its incident edges, we can extend $\phi$ to $v$ since there are at most $\deg(v) < k-1$ colour constraints for $v$, thus $\cH$ has chromatic number less than $k$, a contradiction.
\end{proof}

Using $d$-uniform hypergraphs of large girth and chromatic number, we construct graphs of large minimum degree and girth with total domatic number $1$.

\begin{prop}\label{prop:large-girth}
    For every $d\ge 2$ and $g\ge 3$, there exists a bipartite graph $G$ of minimum degree $d$ and girth at least $g$ such that $\tdom(G) = 1$.
\end{prop}
\begin{proof}
    We take a $d$-uniform hypergraph of girth $g$ and chromatic number at least $d+1$ as in the statement of \Cref{thm:hypergraph-girth-chromatic}, and we extract a vertex-critical sub-hypergraph $\cH$ with chromatic number $d+1$. 
    Clearly, $\cH$ also has girth at least $g$.
    Let $G$ be the bipartite incidence graph of $\cH$.
    Since $\cH$ has edge size and minimum degree both at least $d$, $G$ has minimum degree $d$ and girth at least $2g$.
    We claim that $\tdom(G) = 1$.
    Indeed, attempting to colour $V(\cH)$ with $2$ colours results in at least one monochromatic edge $e$, thus $N_G(e)$ only contains one colour, and $e$ is not dominated by the other colour.
\end{proof}

\begin{rk}
    It is easy to modify the previous proof so as to have $\dom(G) = 2$ instead of $\tdom(G)=1$.
\end{rk}

In the following proposition, we show that the $\dom$-binding function for star-free graphs in \Cref{thm:main} is optimal up to a multiplicative factor.

\begin{prop}\label{prop:alpha2}
    There exist graphs of independence number $2$ with arbitrarily large minimum degree $\delta$, and domatic number at most $(1 + o(1))\frac{\delta}{\log\delta}$.
\end{prop}

\begin{proof}
    Let $n$ be a large integer, $t \coloneqq \ceil{n/(\log n)^3 + 1}$, and $p \coloneqq 1/\log n$. Let $G$ be constructed as follows: The set of vertices of $G$ is $V(G) = A \sqcup B$, where $A = \{u_1 ,\dots, u_n\}$ and $B = \{v_1,\dots,v_t\}$ both induce cliques in $G$; thus $G$ has independence number at most $2$. For every $(i,j) \in [n] \times [t]$, the edge $u_iv_j$ is added to $G$ independently with probability $p$.

    We claim that with nonzero probability, $G$ satisfies the following:
    \begin{itemize}
        \item $G$ has minimum degree $\delta \ge np = \frac{n}{\log n}$, hence $\delta \le n \le (1+o(1))\,\delta\log\delta$
        \item There is no dominating set $X \subseteq A$ of size less than $k \coloneqq \frac{1-p}{p}\log\frac{t}{(\log n)^4}  \sim (\log n)^2$.
    \end{itemize}

    We first examine the first desired property. Every vertex in $A$ has degree at least $n-1$. As for vertices in $B$, they have $t-1$ neighbours in $B$, and their degree in $A$ is determined by a binomial random variable of parameters $(n,p)$. By \Cref{thm:Chernoff}, the probability that $v \in B$ has less than $(1-1/(\log n)^2)\,np$ neighbours in $A$ is at most $e^{-\frac{np}{2(\log n)^4}} \le \frac{1}{n^2}$. By union bound, the probability that at least one vertex has degree at most $(1 - 1/(\log n)^2)\,np$ is at most $t/n^2 \le \frac{1}{n}$. Therefore, with high probability, every vertex in $B$ has degree at least $np - \frac{np}{(\log n)^2} + t-1 \ge np$.

    We now examine the second desired property. Let us fix any set $X \subseteq A$ of size $k$. Independently for any given vertex $v \in B$, the probability that $X$ does not dominate $v$ is $(1-p)^k \ge e^{-\frac{p}{1-p}k} = \frac{(\log n)^4}{t}$, where we used the inequality $1-x \ge e^{-\frac{x}{1-x}}$ which holds for all $x<1$. Therefore, the probability that $X$ dominates $B$ is at most $(1 - (1-p)^k)^t \le e^{-(\log n)^4}$. There are $\binom{n}{k} \le e^{k \log n}$ ways to choose such a set $X$, therefore by union bound, the probability that there exists a set $X \subseteq A$ of size $k$ which dominates $B$ is at most $e^{k \log n - (\log n)^4} \le \frac{1}{n}$.
    
    In conclusion, with nonzero probability, both desired properties hold for $G$; in that case, $G$ satisfies the statement of \Cref{prop:alpha2}. Indeed, in every domatic partition of $G$, at most $t$ dominating sets intersect $B$; the others (which must be contained in $A$) have size at least $k$, and thus $\dom(G) \le t + \frac{n}{k} \le (1+o(1))\frac{n}{(\log n)^2} \le (1 + o(1))\frac{\delta}{\log \delta}$.
\end{proof}

\section{Forbidding a disconnected induced subgraph}\label{sec:disconnected}
We first observe that $\forb(2K_2)$ is not DOM-bounded, since it contains the class of split graphs. 
We infer that $\forb(H)$ can be DOM-bounded only if $H = H_0 + rK_1$ for some connected subgraph $H_0$ and some nonnegative integer $r$. 
In this section, we show that $\forb(H)$ is DOM-bounded if and only if $\forb(H_0)$ is.
We will need the following lemma.

\begin{lemma}
\label{lem:domatic-partition-excluded-set}
    Let $G$ be a graph of minimum degree $\delta$ and let $Z \subseteq V(G)$. There is a partition of $V(G)\setminus Z$ into at least $k-|Z|$ dominating sets of $G$, where 
    \[ k \coloneqq \min \left\{\sqrt{\delta/2}, \dom(G\setminus Z) \right\}.\]
\end{lemma}

\begin{proof}
    Let us write $Z = \{z_1, \ldots, z_t\}$ with $t\leq k$, otherwise there is nothing to prove.
    Let $D_1, \ldots, D_k$ be a domatic partition of $G\setminus Z$.
    We claim that each $z_i\in Z$ has a set $S(z_i)$ of private neighbours such that
    \begin{itemize}
        \item $|S(z_i)| = \floor{\frac{\delta}{2k}}$, 
        \item there exists $j\in [k]$ such that $S(z_i) \subseteq D_j$, and 
        \item the sets $\{S(z_i)\}_{i\in [t]}$ are pairwise disjoint. 
    \end{itemize}
    Those can be constructed iteratively; the set $S(z_i)$ is selected from the colour class $D_j$ which has the largest intersection with $N(z_i) \setminus \bigcup_{i'<i} S(z_{i'})$. 
    By the Pigeonhole Principle, that intersection has size at least $\frac{1}{k}\pth{\deg(z_i) - (i-1)\frac{\delta}{2k}} \ge \frac{1}{k}\pth{\delta - (t-1) \frac{\delta}{2k}} \ge \frac{\delta}{2k}$. 
    Note that, by definition of $k$, $\frac{\delta}{2k} \ge k$.
    Without loss of generality, let us assume that $\bigcup_{z\in Z} S(z) \subseteq \bigcup_{j=k-t+1}^t D_j$.
    For each $j\in [k-t]$ we construct $D'_j$ from $D_j$ by adding one private vertex from each $S(z_i)$ which is possible since $|S(z_i)| > k-t$.
    By construction, each $D'_j$ is a dominating set of $G$, disjoint from the others. 
    So $(D'_j)_{j\in [k-t]}$ is the desired domatic partition of $G$ up to assigning the remaining vertices arbitrarily.  
\end{proof}

\begin{thm} \label{thm:disconnected:full}
Let $r$ be a nonnegative integer and let $\delta$ be a positive integer.
Let $H_0$ be a graph such that $\forb(H_0)$ is DOM-bounded, with DOM-binding function $f_0$, and let $H$ be obtained by adding to $H_0$ a disjoint independent set on $r$ vertices. Let $t\coloneqq |V(H)|$.
For every $H$-free graph $G$ of minimum degree $\delta$, one has
\[ \dom(G) \ge  \floor{\log_2 \frac{\min \big\{f_0(\delta/2), \sqrt{\delta}/2\big\}}{2t}}.\]
\end{thm}

\begin{proof}
    Let $k\coloneqq \floor{\log_2 \frac{\min \big\{f_0(\delta/2), \sqrt{\delta}/2\big\}}{2t}}$.
    
    \noindent
    We construct a domatic partition of $G$ of size at least $k$ with the following procedure. 
    \begin{enumerate}
        \item Initially set $G_0\gets G$, $Z\gets \emptyset$, and $i\gets 0$.
        \item While $i<k$ and $G_i$ is not $H_0$-free, let $X_i$ be such that $G_i[X_i]$ is isomorphic to $H_0$. Let $I_i$ be a maximal independent set of $G_i \setminus N[X_i]$, and initially set $D_i \gets X_i \cup I_i$.
        For every $z\in Z$, choose a vertex $y_z\in N(z)\setminus Z$ and add $y_z$ to $D_i$. 
        Set $G_{i+1} \gets G_i \setminus D_i$, $Z\gets Z\cup D_i$, and $i\gets i+1$. 
        \item Return the domatic partition formed by $D_0, \ldots, D_{i-1}$ together with that of $G_i = G \setminus Z$ promised by \Cref{lem:domatic-partition-excluded-set} applied to $G$ and $Z$.
    \end{enumerate}

    \begin{claim}
        For every nonnegative integer $i$, at the end of iteration $i$ of the while loop, $D_i$ is a dominating set of $G$ of size at most $2^i t$. 
    \end{claim}

    \begin{proofclaim}
        Since $G$ (and so, in particular, also $G_i$) is $H$-free and $G_i[X_i]\cong H_0$, we infer that $G_i\setminus N[X_i]$ is $rK_1$-free; said otherwise $\alpha(G_i\setminus N[X_i])<r$. So $|X_i \cup I_i|<t$. $X_i$ dominates $N[X_i]$, and by maximality $I_i$ dominates $G_i\setminus N[X_i]$, so $X_i \cup I_i$ dominates $G_i$.
        This settles the proof when $i=0$. 
        When $i \ge 1$, we may assume by induction that, for each $j<i$, $|D_j| \le 2^jt$. At the beginning of the iteration, we have $Z = \bigcup_{j<i}D_j$, so $Z \le \sum_{j<i} 2^jt = (2^i-1)t$. Note that, by the assumption that $i< k \le \log_2 \frac{\sqrt{\delta}}{4t}$, one has $|Z|<\delta$, which ensures that $N(z)\setminus Z \neq \emptyset$ for all $z\in Z$. 
        Each $z\in Z$ adds at most one new vertex into $D_i$, so at the end of the iteration $|D_i|\le |X_i \cup I_i| + |Z| \le 2^it$, as desired.
        By construction, each $z\in Z$ has a neighbour in $D_i$, and $D_i$ already dominates $V(G_i) = V(G)\setminus Z$, so $D_i$ is a dominating set of $G$.
    \end{proofclaim}

    Let $i$ be the total number of iterations of the while loop. If $i=k$ then we are done, so we assume that $i<k$, which implies that $G_i$ is $H_0$-free. 
    We have $|Z| \le (2^i-1)t < \frac{\min \big\{f_0(\delta/2), \sqrt{\delta}/2\big\}}{2} < \frac{\delta}{2}$, and so $\delta(G_i) \ge \delta - |Z| \ge \delta/2$. 
    The returned domatic partition has size at least
    \begin{align*}
        i + \min \left\{f_0(\delta/2), \sqrt{\delta}/2 \right\} - |Z| \ge \frac{1}{2} \min \left\{f_0(\delta/2), \sqrt{\delta}/2 \right\} > k.
    \end{align*}
    This ends the proof.
\end{proof}

\Cref{thm:disconnected} is a corollary of \Cref{thm:disconnected:full}.

\section{Line graphs of rank-\texorpdfstring{$k$}{k} hypergraphs}\label{sec:linegraphs}

A hypergraph has rank $k$ if every edge has size at most $k$. The line graph of a hypergraph is the intersection graph of its edges. Line graphs of rank-$k$ hypergraphs are in particular induced-$K_{1,k+1}$-free.
Indeed, if an edge $e$ of the hypergraph intersects $k+1$ pairwise disjoint edges, then two of these intersections must occur at the same vertex of $e$, forcing the corresponding two neighbours of $e$ in the line graph to be adjacent.

In this section, we show that line graphs of rank-$k$ hypergraphs are linearly DOM-bounded. We shall first use the following result due to Lovász \cite{lovasz1970generalization} concerning the panchromatic number of hypergraphs.

\begin{thm}[Lovász, 1970]\label{thm:panchromatic}
	Let $\mathcal{H} = (V,\mathcal{E})$ be a hypergraph. If there exists an integer $\beta \ge 1$ such that every nonempty subset $\mathcal{F} \subseteq \mathcal{E}$ verifies $$\card{\UF} > (\beta - 1)|\mathcal{F}|,$$
	then $\mathcal{H}$ has panchromatic number at least $\beta$.
\end{thm}

\Cref{thm:panchromatic} suffices to show that the class of line graphs of rank-$k$ hypergraphs is linearly DOM-bounded.

\begin{cor}\label{cor:rank-k-linegraph-loose}
	Let $G$ be a line graph of some rank-$k$ hypergraph.
	Let $\delta$ be the minimum degree of $G$. Then \[\dom(G) \ge \frac{\delta}{k^2}.\]
\end{cor}

\begin{proof}
	Let $H$ be a rank-$k$ hypergraph such that $G$ is the line graph of $H$. For each $v \in V(H)$, we denote $E_v$ the set of edges of $H$ incident to $v$.
	Let $V^+ \coloneqq \sst{v \in V(H)}{|E_v|\ge \delta/k}$; clearly, $V^+$ is a transversal of $H$. Consider the hypergraph $\mathcal{H}$ with vertex-set $\cE(H)$ and hyperedge-set $\cE(\cH) \coloneqq \sst{E_v}{v \in V^+}$. Observe that each edge $e \in \cE(H)$ is contained in $|e|\le k$ distinct sets of the form $E_v$, one for each vertex of $e$. 
	
	Let $\mathcal{F}$ be a non-empty subset of $\cE(\cH)$. By double counting, we have $$\frac{\delta}{k}|\mathcal{F}| \le \sum_{F\in \mathcal{F}}|F| = \sum_{e \in \UF} |\{F\in \mathcal{F} : e\in F\}| \le k \card{\UF},$$
	
	hence $\card{\UF} \ge \frac{\delta}{k^2}|\mathcal{F}|$ for all non-empty subsets $\mathcal{F} \subseteq \cE(\cH)$. Thus, by \Cref{thm:panchromatic}, $\mathcal{H}$ has panchromatic number at least $\frac{\delta}{k^2}$. Consider any panchromatic $\frac{\delta}{k^2}$-colouring of $\mathcal{H}$, and observe that this same colouring of $\cE(H)$ corresponds to a dominating colouring in $G$: indeed, for every edge $e \in \cE(H)$, $e$ is incident to some vertex $v \in V^+$, and thus is incident to all edges in the set $E_v$, which contains every colour. Therefore $\dom(G) \ge \delta / k^2$.
\end{proof}

In the rest of this subsection, we show that \Cref{cor:rank-k-linegraph-loose} can be quantitatively improved, starting with line graphs of simple graphs ($2$-uniform hypergraphs).

\subsection{Line graphs of simple graphs}

A \emph{cover decomposition} of a graph is a partition of its edges into edge covers (a subset of edges which covers all vertices). The \emph{cover index} of a graph $G$ is the maximum size of a cover decomposition of $G$. Clearly, a graph of minimum degree $\delta$ has cover index at most $\delta$. In his thesis \cite{gupta1967studies} (see also \cite{gupta1974decompositions}), Gupta proved that simple graphs of minimum degree $\delta$ have cover index at least $\delta-1$.

\begin{thm}[Gupta, 1967]\label{thm:gupta}
	Let $G$ be a simple graph. Then $G$ has a cover decomposition of size $\delta(G)-1$.
\end{thm}

Using this result on cover decompositions, we improve the lower bound $\dom(G) \ge \ceil{\delta(G)/4}$ from \Cref{cor:rank-k-linegraph-loose} to $\dom(G) \ge \ceil{\delta(G)/2} - 1$, when considering line graphs of simple graphs. 

\begin{thm}\label{thm:linegraph-lower}
	Let $G$ be the line graph of a simple graph, and $\delta \coloneqq \delta(G)$. Then \[\dom(G) \ge \lceil \delta/2\rceil -1.\]
\end{thm}
\begin{proof}
	Let $H$ be a simple graph such that $L(H) = G$. Let $V^+\subseteq V(H)$ be the vertices of $H$ of degree at least $\lceil\delta/2\rceil$; observe that $V^+$ is a vertex cover of $H$. Consider the graph $H'$ which is a supergraph of $H$ where we glue a clique of size $\delta$ to every vertex of $V(H)\setminus V^+$. Since $H'$ has minimum degree at least $\lceil\delta/2\rceil$, we apply \Cref{thm:gupta} to obtain a cover decomposition $E'_1 \sqcup \dots \sqcup E'_r$ of $H'$, where $r \ge \lceil\delta/2\rceil-1$. For $i \in [r]$, let $E_i \coloneqq E'_i \cap E(H)$; then $E_i$ covers $V^+$, since no vertex in $V^+$ is incident to an edge in $E(H')\setminus E(H)$, by construction. We conclude that the decomposition $E_1 \sqcup \dots \sqcup E_r$ is a domatic partition of $L(H) = G$.
\end{proof}
\begin{rk}
	\Cref{thm:linegraph-lower} is nearly tight: consider the line graph of the complete bipartite graph $K_{n,n}$; one easily shows that $L(K_{n,n})$ has degree $2(n-1)$ and has domatic number $n$, thus we obtain $\dom(L(K_{n,n})) \le \delta/2 + 1$.
\end{rk}

\subsection{Improved bounds via cover decompositions}

The notion of cover decomposition extends naturally to hypergraphs. This was considered by Bollob\'{a}s  et al. in \cite{bollobas2013cover}, where they establish an asymptotically tight lower bound for the size of a cover decomposition in a hypergraph of given rank, as a function of its minimum degree. Note that this result does not require the hypergraph to be simple or loopless (hyperedges of size $1$ are allowed).

\begin{thm}[Bollob\'{a}s, Pritchard, Rothvo\ss, Scott, 2013]\label{thm:cover-decomposition-hypergraphs}
	Let $H$ be a rank-$k$ hypergraph. Then $H$ has a cover decomposition of size $\frac{\delta(H)}{\log k + \mathcal{O}(\log\log k)}$.
\end{thm}

Using the same strategy as in \Cref{thm:linegraph-lower}, we deduce the following bound on the domatic number of line graphs of rank-$k$ hypergraphs, improving \Cref{cor:rank-k-linegraph-loose} by a factor $(1+o(1))\frac{k}{\log k}$. The following is a restatement of \Cref{thm:linegraph}.

\begin{thm}
	\label{thm:linegraph-hypergraph-improved}
	Let $G$ be the line graph of some rank-$k$ (multi)hypergraph, and $\delta \coloneqq \delta(G)$. Then 
	\[\dom(G) \ge \frac{\delta}{k(\log k + O(\log \log k))}.\]
\end{thm}

\begin{proof}
	Let $H$ be a hypergraph such that $L(H) = G$. Let $V^+\subseteq V(H)$ be the vertices of $H$ of degree at least $\lceil\delta/k\rceil$; observe that $V^+$ is a vertex cover of $H$. Consider the hypergraph $H'$ obtained from $H$ by adding $\delta$ copies of $\{v\}$ as an edge for each vertex of $v\in V(H)\setminus V^+$. By construction, $H'$ has minimum degree at least $\lceil\delta/k\rceil$, so we can apply \Cref{thm:cover-decomposition-hypergraphs} to obtain a cover decomposition $E'_1 \sqcup \dots \sqcup E'_r$ of $H'$, where $r \ge \lceil\delta/k\rceil/(\log k + O(\log \log k))$. For $i \in [r]$, let $E_i \coloneqq E'_i \cap E(H)$; then $E_i$ covers $V^+$, since no vertex in $V^+$ is incident to an edge in $E(H')\setminus E(H)$, by construction. We conclude that the decomposition $E_1 \sqcup \dots \sqcup E_r$ is a domatic partition of $L(H) = G$.
\end{proof}
\begin{rk}
	\Cref{thm:linegraph-hypergraph-improved} is tight up to a $(1 + o(1))\log k$ factor; the complete $k$-partite $k$-uniform hypergraph $\cH$ with all parts of size $n$ has edge-degree $\delta_e \ge k(n-1)^{k-1}$ and has edge domination number at least $n$, therefore $\dom(L(\cH)) \le \frac{n^k}{n} =  (1-o(1))\frac{\delta_e}{k}$.
\end{rk}

\section{Cographs}
\label{sec:cographs}

The class of \emph{cographs} corresponds to the class of $P_4$-free graphs. Every cograph can be constructed inductively by taking either the disjoint union or the complete join of two smaller cographs, the base case being a single vertex. As a consequence, a cograph can be represented by a \emph{cotree}, that is a binary tree whose internal nodes are labelled either $\wedge$ to represent a complete join of the cographs represented by their children, or $\vee$ to represent the disjoint union of the cographs represented by their children. The leaves of the cotree are the vertices of the cograph it represents.

We begin with a tight bound for $\tdom$ of cographs, with equality when $G$ is a clique of odd order.

\begin{thm}
    \label{thm:tdom-cographs}
    Let $G$ be a cograph with minimum degree $\delta \geq 1$.
    We have
    \[ \tdom(G) \ge \frac{\delta(G)}{2}.\]
\end{thm}

\begin{proof}
We prove the result by induction on $G$.
When $G$ is a single vertex, we have $\delta(G)=0$ and $\tdom(G)=0$, so the result holds.
If $G$ is not connected, let $C$ be a connected component with the smallest value of $\tdom(G[C])$. One has
\[ \tdom(G) = \tdom(G[C]) \ge \frac{\delta(G[C])}{2} \ge \frac{\delta(G)}{2},\]
as desired. 
Finally, if $G = G_0 \wedge G_1$ with $n_0 \coloneqq |V(G_0)|\ge n_1 \coloneqq |V(G_1)|$, then 
\[ \delta(G) = \min\{\delta(G_0)+n_1, \delta(G_1)+n_0\}.\]

For all $i\in \{0,1\}$, every total dominating set of $G_i$ is also a total dominating set of $G$, so we have $\tdom(G) \ge \tdom(G_0) + \tdom(G_1)$. If we have $\tdom(G_0) = n_0/2$, then by the induction hypothesis we obtain 
\[ \tdom(G) \ge \frac{n_0}{2} + \frac{\delta(G_1)}{2} \ge \frac{\delta(G)}{2}.\]
We now assume that $\tdom(G_0)<n_0/2$. 
Let $\cP_0$ be a partition of $V(G_0)$ into $\tdom(G_0)$  total dominating sets; every part of $\cP_0$ has size at least $2$ (since total dominating sets have size at least $2$), and since $|\cP_0|<n_0/2$ at least one part has size at least $3$.
Let $\cQ \subseteq \cP_0$ be a subset of minimum size such that $\sum_{Q \in \cQ} |Q| \ge n_1$. We have $|\cQ|\le \ceil{(n_1-1)/2} \le n_1/2$. Let $M$ be a matching between $V(G_1)$ and $\bigcup \cQ$ that saturates $V(G_1)$. Every $\{u,v\}\in M$ is a total dominating set of $G$, so $\cP \coloneqq (\cP_0 \setminus \cQ) \cup M$ is a partition of $V(G)$ into total dominating sets, of size
\[
    |\cP| \ge \tdom(G_0)-|\cQ| + n_1 \ge \frac{\delta(G_0)}{2} + \frac{n_1}{2} \ge \frac{\delta(G)}{2}. \qedhere
\]
\end{proof}

We slightly adapt the above proof to derive a tight bound for the domatic number of cographs. We first need to reduce the universal vertices from a given graph $G$ (a vertex $v\in V(G)$ is \emph{universal in $G$} if $N_G[v]=V(G)$).

We begin with the following straightforward observation.

\begin{obs}
    \label{obs:dom-vertex-removal}
    Let $G$ be a graph, and let $v\in V(G)$. Then
    \[ \dom(G) \le 1 + \dom(G \setminus v). \]
\end{obs}

\begin{proof}
    Let $\cP$ be a domatic partition of $G$ of size $\dom(G)$, and let $D\in\cP$ be the part containing $v$. Remove $D$ from the partition. Then assign each vertex of $D\setminus\{v\}$ arbitrarily to one of the remaining parts.
    This is a domatic partition of $G\setminus v$, hence $\dom(G\setminus v) \ge \dom(G)-1$, as desired.
\end{proof}

\begin{lemma}\label{lem:cograph-universal-vertices}
Let $G$ be a graph, and let $\cS$ be the set of universal vertices of $G$.
Then $\dom(G) = |\cS| + \dom(G\setminus \cS)$.
\end{lemma}

\begin{proof}
We have the upper bound $\dom(G) \le |\cS| + \dom(G\setminus \cS)$ by repeating $|\cS|$ times \Cref{obs:dom-vertex-removal}. Let us now prove the lower bound.

Let $\cP_0$ be a partition of $G\setminus \cS$ into $\dom(G\setminus \cS)$ dominating sets of $G\setminus\cS$, and let $\cP \coloneqq \cP_0\cup \sst{\{s\}}{s\in\cS}$.
This is a domatic partition of $G$, hence $\dom(G) \ge \dom(G\setminus \cS) + |\cS|$, as desired.
\end{proof}

\begin{thm}
\label{thm:dom-cographs}
For every cograph $G$ with minimum degree $\delta$,
\[
\dom(G) \geq 1+\frac{\delta}2.
\]    
\end{thm}

\begin{proof}
We proceed by induction on the number of vertices of $G$.
When $G$ has a single vertex, we have $\delta(G)=0$ and $\dom(G)=1$, so the result holds.
If $G$ is not connected, let $C$ be its connected component with the smallest value of $\dom(G[C])$. One has
\[
    \dom(G) = \dom(G[C]) \ge \frac{\delta(G[C])}{2}+1 \ge \frac{\delta(G)}{2}+1,
\]
as desired. 
If $G$ has at least one universal vertex, let $\cS$ be the set of universal vertices of $G$. By \Cref{lem:cograph-universal-vertices}, we have
\[
    \dom(G)=|\cS|+\dom(G\setminus \cS).
\]
If $G\setminus \cS$ is empty, then $G$ is complete, and the result is immediate. Otherwise, by the induction hypothesis applied to $G\setminus \cS$, we obtain
\[
    \dom(G)
    \ge |\cS|+1+\frac{\delta(G\setminus \cS)}{2}.
\]
Since $G\setminus \cS$ is nonempty, we have
\[
    \delta(G)=|\cS|+\delta(G\setminus \cS).
\]
Therefore,
\[
    \dom(G)
    \ge |\cS|+1+\frac{\delta(G)-|\cS|}{2}
    = 1+\frac{\delta(G)}{2}+\frac{|\cS|}{2}
    \ge 1+\frac{\delta(G)}{2}.
\]
We now assume that $G$ has no universal vertex, and that $G = G_0 \wedge G_1$ with $n_0 \coloneqq |V(G_0)|\ge n_1 \coloneqq |V(G_1)|$. We have
\[
    \delta(G) = \min\{\delta(G_0)+n_1, \delta(G_1)+n_0\}.
\]

For all $i\in \{0,1\}$, every dominating set of $G_i$ is also a dominating set of $G$, so we have $\dom(G) \ge \dom(G_0) + \dom(G_1)$. If we have $\dom(G_0) = n_0/2$, then by the induction hypothesis we obtain 
\[ \dom(G) \ge \frac{n_0}{2} + \frac{\delta(G_1)}{2}+1 \ge \frac{\delta(G)}{2}+1.\]
We now assume that $\dom(G_0)<n_0/2$. 
Let $\cP_0$ be a partition of $V(G_0)$ into $\dom(G_0)$ dominating sets; every part of $\cP_0$ has size at least $2$ (since $G$ and therefore also $G_0$ has no universal vertex), and since $|\cP_0|<n_0/2$ at least one part has size at least $3$.
Let $\cQ \subseteq \cP_0$ be a subset of minimum size such that $\sum_{Q \in \cQ} |Q| \ge n_1$. We have $|\cQ|\le \ceil{(n_1-1)/2} \le n_1/2$. Let $M$ be a matching between $V(G_1)$ and $\bigcup \cQ$ that saturates $V(G_1)$. Every $\{u,v\}\in M$ is a dominating set of $G$, so $\cP \coloneqq (\cP_0 \setminus \cQ) \cup M$ is a partition of $V(G)$ into dominating sets, of size
\[
    |\cP| \ge \dom(G_0)-|\cQ| + n_1 \ge \frac{\delta(G_0)}{2} + 1 + \frac{n_1}{2} \ge \frac{\delta(G)}{2} + 1. \qedhere
\]
\end{proof}

Observe that the bound given by \Cref{thm:dom-cographs} is tight, as equality holds when $G$ is a complete multipartite graph with all parts of size $2$.

\section{Forbidding stars}
\label{sec:stars}

In this section, we denote by $N^2_G[u]$ the set of vertices at distance at most 2 from a vertex $u$ in a graph $G$. We omit $G$ when it is clear from context.

We show the existence of dominating colourings of graphs $G$ with some well-chosen structural properties through an analysis of the uniformly random $k$-colouring of $G$. This strategy has been successfully applied in \cite{FHKS02} to obtain the following lower bound, which in particular demonstrates that the class of regular graphs is DOM-bounded.

\begin{thm}[Feige, Halld\'orsson, Kortsarz, Srinivasan, 2002]
    For every graph $G$,
    \[ \dom(G) \ge (1+o(1)) \frac{\delta(G)}{\log \Delta(G)}. \]
\end{thm}

Using a similar approach, relying on increasingly technical applications of the LLL, we show that unit disk graphs and, more generally, star-free graphs are DOM-bounded.

\subsection{Unit disk graphs}
We first consider unit disk graphs as a special case of star-free graphs (they are contained in $\forb(K_{1,6})$, cf. \cite[Lemma 3.2]{MBBRR95}). A key structural property of unit disk graphs is that every second neighbourhood can be covered with a finite number of cliques; we will rely of that property in order to apply the LLL. This is a direct consequence of the following geometrical fact.

\begin{figure}[!ht]
    \centering
    \begin{tikzpicture}[line cap=round,line join=round,> = stealth, line width=1.5pt, x=2.5cm,y=2.5cm]

    \fill  (-5.196152422706632,-6) circle (2pt);
    \draw  (-5.196152422706632,-6) circle (2);
    
    \draw  (-5.413060595501682,-6.450500659904859)-- (-4.979165626017703,-6.4504627954256355);
    \draw  (-4.979165626017703,-6.4504627954256355)-- (-4.708674599964651,-6.111199695749307);
    \draw  (-4.708674599964651,-6.111199695749307)-- (-4.8052833227749225,-5.68819662170115);
    \draw  (-4.8052833227749225,-5.68819662170115)-- (-5.196152422706632,-5.5);
    \draw  (-5.196152422706632,-5.5)-- (-5.58707593664049,-5.688264845977223);
    \draw  (-5.58707593664049,-5.688264845977223)-- (-5.683610830015819,-6.111284774985581);
    \draw  (-5.683610830015819,-6.111284774985581)-- (-5.413060595501682,-6.450500659904859);
    \draw  (-4.905644600183051,-7.272249594848275)-- (-4.750928245646336,-7.9498142045233395);
    \draw  (-5.486282960025566,-7.272131798052397)-- (-5.641036286210788,-7.94989188110357);
    \draw  (-4.738592656396453,-6.94989029079898)-- (-4.905644600183051,-7.272249594848275);
    \draw  (-5.653443664069021,-6.949968855578759)-- (-5.486282960025566,-7.272131798052397);
    \draw  (-6.009521024899024,-7.020258083657558)-- (-5.653443664069021,-6.949968855578759);
    \draw  (-6.371782318903132,-6.566487575952188)-- (-6.224014337913103,-6.234845454646383);
    \draw  (-6.224014337913103,-6.234845454646383)-- (-6.501148110623531,-6.0002847058981015);
    \draw  (-6.3720293838179884,-5.4340254432691975)-- (-6.020607028332221,-5.342788384488842);
    \draw  (-6.020607028332221,-5.342788384488842)-- (-6.010001082710506,-4.979870505695979);
    \draw  (-5.4867712688456445,-4.7277757615860585)-- (-5.196336441378032,-4.945650688605946);
    \draw  (-5.196336441378032,-4.945650688605946)-- (-4.905977684299886,-4.727674393981768);
    \draw  (-4.38265990464755,-4.979586481076931)-- (-4.371927277436555,-5.342500635586233);
    \draw  (-4.371927277436555,-5.342500635586233)-- (-4.02047309561979,-5.43361501926719);
    \draw  (-3.8911567149135426,-5.999829176460321)-- (-4.1682085935943896,-6.234486648852448);
    \draw  (-4.1682085935943896,-6.234486648852448)-- (-4.020324856670155,-6.566077169198511);
    \draw  (-4.3823927888939505,-7.0202005121132744)-- (-4.738592656396453,-6.94989029079898);
    \draw  (-6.009521024899024,-7.020258083657558)-- (-6.442751864003199,-7.5639660587618565);
    \draw  (-6.371782318903132,-6.566487575952188)-- (-6.997889974951227,-6.868183041087336);
    \draw  (-6.501148110623531,-6.0002847058981015)-- (-7.196152375110162,-6.0004363323096985);
    \draw  (-6.3720293838179884,-5.4340254432691975)-- (-6.998268619741076,-5.132603197846493);
    \draw  (-6.010001082710506,-4.979870505695979)-- (-6.44343415424221,-4.436578021717227);
    \draw  (-5.4867712688456445,-4.7277757615860585)-- (-5.641546751459857,-4.050224656039967);
    \draw  (-4.905977684299886,-4.727674393981768)-- (-4.751438721063227,-4.050069302880067);
    \draw  (-4.38265990464755,-4.979586481076931)-- (-3.949416504371215,-4.436142733516786);
    \draw  (-4.02047309561979,-5.43361501926719)-- (-3.394339114059685,-5.131974193481128);
    \draw  (-3.8911567149135426,-5.999829176460321)-- (-3.1961524398413586,-5.999738200612927);
    \draw  (-4.020324856670155,-6.566077169198511)-- (-3.39411192718469,-6.867554063156318);
    \draw  (-4.3823927888939505,-7.0202005121132744)-- (-3.949007130225851,-7.563530818194201);
    \draw  (-5.653443664069021,-6.949968855578759)-- (-5.413060595501682,-6.450500659904859);
    \draw  (-4.979165626017703,-6.4504627954256355)-- (-4.738592656396453,-6.94989029079898);
    \draw  (-4.708674599964651,-6.111199695749307)-- (-4.1682085935943896,-6.234486648852448);
    \draw  (-4.8052833227749225,-5.68819662170115)-- (-4.371927277436555,-5.342500635586233);
    \draw  (-5.196152422706632,-5.5)-- (-5.196336441378032,-4.945650688605946);
    \draw  (-6.020607028332221,-5.342788384488842)-- (-5.58707593664049,-5.688264845977223);
    \draw  (-6.224014337913103,-6.234845454646383)-- (-5.683610830015819,-6.111284774985581);
    \draw  (-6.010001082710506,-4.979870505695979)-- (-5.4867712688456445,-4.7277757615860585);
    \draw  (-4.905977684299886,-4.727674393981768)-- (-4.38265990464755,-4.979586481076931);
    \draw  (-4.02047309561979,-5.43361501926719)-- (-3.8911567149135426,-5.999829176460321);
    \draw  (-4.020324856670155,-6.566077169198511)-- (-4.3823927888939505,-7.0202005121132744);
    \draw  (-4.905644600183051,-7.272249594848275)-- (-5.486282960025566,-7.272131798052397);
    \draw  (-6.009521024899024,-7.020258083657558)-- (-6.371782318903132,-6.566487575952188);
    \draw  (-6.501148110623531,-6.0002847058981015)-- (-6.3720293838179884,-5.4340254432691975);
    \draw [<->, line width=1pt,dashed] (-5.641036286210788,-7.94989188110357)-- (-4.905644600183062,-7.27224959484827);
    \draw [<->,line width=1pt,dashed] (-4.750928245646336,-7.9498142045233395)-- (-4.738592656396468,-6.949890290798982);
    \draw [<->,line width=1pt,dashed] (-5.413060595501683,-6.450500659904855)-- (-5.196152422706632,-6);
    
    \begin{scriptsize}
    \draw (-5.167759915974884,-7.609679163788004) node {1};
    \draw (-4.600383772962272,-7.362361357859456) node {1};
    \draw (-5.11684154416606,-6.1912388062566235) node {0.5};
    \end{scriptsize}
    \end{tikzpicture}
    \caption{A partition of the disk of radius $2$ into $22$ parts of diameter $\le 1$.}
    \label{fig:partition-diameter1}
\end{figure}
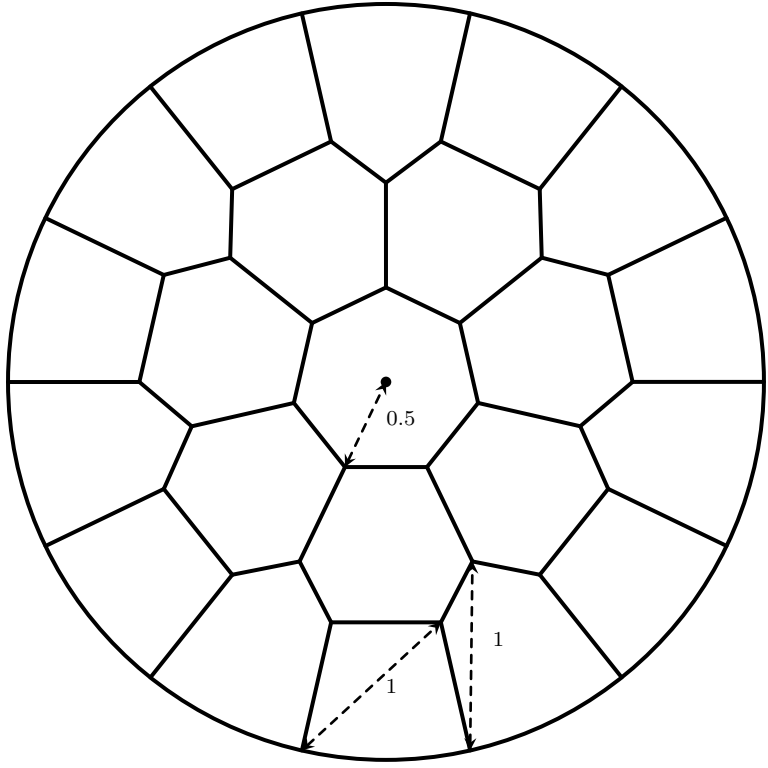

\begin{lemma}
    \label{lem:unit-disk}
    In the Euclidean plane, there exists a partition of the disk of radius $2$ into $22$ parts of diameter at most $1$. 
\end{lemma}

\begin{proof}
    The partition is given in \Cref{fig:partition-diameter1}. There are $14$ points uniformly distributed along the bounding circle; the distance between two consecutive ones is therefore less than $2\pi/7 < 1$.
    The central heptagon is inscribed in a circle of radius $r=1/2$.
\end{proof}

\begin{thm}
    \label{thm:unit-disk}
    For every unit disk graph $G$ of minimum degree $\delta \ge e^{44}$, 
    \[\dom(G) \ge \floor{\frac{\delta}{2\log \delta}}.\]
    
\end{thm}

\begin{proof}
    Let $k\coloneqq \floor{\frac{\delta}{2\log{\delta}}}$.
    For each vertex $v\in V(G)$, we draw a colour $\bsig(v) \in [k]$ independently and uniformly at random.
    For each colour $i\in [k]$ and vertex $u\in V(G)$, we let $B_{i,u}$ be the random bad event that there is no vertex of colour $i$ in $N[u]$.
    We show with an application of \Cref{cor:LLL} that, with non-zero probability, no bad event occurs, in which case the realisation of $\bsig$ is a dominating $k$-colouring of $G$.
    Let $u\in V(G)$ be a fixed vertex, and $i\in [k]$ a fixed colour.
    We have 
    \begin{equation}
        \pr{B_{i,u}}= \pth{1-\frac{1}{k}}^{\deg(u)+1} \le e^{-\frac{\deg(u)+1}{k}} < 1.
    \end{equation}
    Observe that the bad event $B_{i,u}$ is entirely determined by the random choices of colours of the vertices in $N[u]$. Hence, for any vertex $v\in V(G)$ at distance at least $3$ from $u$ and every colour $j \in [k]$, since $N[u]$ and $N[v]$ are disjoint, the bad events $B_{i,u}$ and $B_{j,v}$ are independent. 
    Therefore, in the dependency graph $\Gamma$ of the bad events $(B_{i,u})_{i\in [k], u\in V(G)}$, we have $N_\Gamma(B_{i,u}) \subseteq \sst{B_{j,v}}{j\in [k], v\in N_G^2[u]}$.
    We observe that, in the embedding of $G$ in the plane, $N^2[u]$ is contained in the disk of radius $2$ centred in $u$. By \Cref{lem:unit-disk}, this disk can be partitioned into $22$ parts of diameter at most $1$, so $N^2[v]$ can be partitioned into $22$ cliques $C_1, \ldots, C_{22}$ --- since the set of vertices included in a region of diameter at most $1$ induces a clique in a unit-disk graph.
    This lets us decompose the sum in the left-hand side of condition \eqref{eq:LLL} from \Cref{cor:LLL} as follows:
    \[
    \sum_{j \in [k], w \in N^2[v]} \pr{B_{j,w}}\le k\sum_{i=1}^{22}\sum_{w \in C_i} e^{-\frac{\deg(w)+1}{k}}
    \]
    We now bound the contribution of each clique $C_i$ in that sum separately. 
    Let $t_i \coloneqq |C_i|$; either $t_i\le \delta$, and 
        \[
        \sum_{w \in C_i} e^{-\frac{\deg(w)+1}{k}}\le \delta e^{-\frac{\delta+1}{k}} \le \frac{1}{\delta};
        \]
    otherwise $t_i \ge \delta$, and
        \[
        \sum_{w \in C_i} e^{-\frac{\deg(w)+1}{k}}\le t_i e^{-\frac{t_i}{k}} \le \delta e^{-\frac{\delta}{k}} \le \frac{1}{\delta},
        \]
        where we have used that $\deg(w) \ge t_i-1$ since $C_i$ is a clique, and that the function $x \mapsto x e^{-\frac{x}{k}}$ is non-increasing when $x\ge k$.
    
    We conclude that
    \[
    \sum_{j \in [k], w \in N^2[v]} \pr{B_{j,w}}\le \frac{22k}{\delta} \le \frac{11}{\log \delta}\le \frac{1}{4},
    \]
    which lets us apply \Cref{cor:LLL} and thus end the proof.
\end{proof}

\subsection{A more abstract setting}

Let $G$ be a graph.
Given a real value $p\in [0,1]$, an integer $K\ge 1$, and a vertex $u\in V(G)$, the \emph{$(p,K)$-constrained neighbourhood $\cstN(u)$ of $u$} is the subset of neighbours $v$ of $u$ such that $N(v)\setminus N(u)$ can be decomposed into $K$ parts $X_1, \ldots, X_K$, each of which satisfies $\min_{w\in X_i} (\deg_G(w)+1) \ge p|X_i|$. We call such a decomposition a \emph{$p$-dense $K$-decomposition} of $N(v)\setminus N(u)$.
Observe that, given a set $X\subseteq V(G)$ and a real value $p\in [0,1]$, one can compute an optimal $p$-dense decomposition of $X$ (one that uses the minimum number $K$ of parts) greedily with the following procedure. A direct consequence is that computing $\cstN(u)$ can be done in $O(\Delta(G)^2)$ time. 

\begin{algorithm}
\caption{A greedy construction of a $p$-dense decomposition}
\begin{algorithmic}[1]
\State let $v_1, \ldots, v_n$ be $X$ ordered by non-increasing degree in $G$
\State $i \gets 1$; $j \gets 1$
\State $X_i \gets \emptyset$
\For{$j \in [n]$}
    \If{$\deg_G(v_j) + 1 \ge p(|X_i| + 1)$}
        \State $X_i \gets X_i \cup \{v_j\}$
    \Else
    \State $i \gets i + 1$
    \State $X_i \gets \{v_j\}$
    \EndIf
\EndFor
\State \Return $(X_1, \ldots, X_i)$
\end{algorithmic}
\end{algorithm}

\begin{thm}
\label{thm:main-technical}
    Let $0 < c, p \le 1$ be real constants and $K\ge 1$ an integer.
    Let $G$ be a graph of minimum degree $\delta \ge 4^{e/9+K/p}$.
    If, for every vertex $u \in V(G)$,
    $|\cstN(u)|\ge c\cdot \deg(u)$, then
    \[ \tdom(G) \ge \floor{\frac{c \delta}{3 \log \delta}}.\]
\end{thm}

\begin{proof}
    Let $k\coloneqq \floor{\frac{c\delta}{3\log{\delta}}}$. By the assumptions, we have $\delta \ge 4$, and we may assume that $c\delta \ge 3$, otherwise $k=0$ and there is nothing to prove. If $k=1$, then colouring every vertex with the unique colour gives a total-dominating $1$-colouring, since $|\cstN(u)|\ge 1$ for every $u\in V(G)$. Hence, we may assume that $k\ge 2$.
    For each vertex $u\in V(G)$, let us draw a colour $\bsig(u) \in [k]$ independently and uniformly at random.
    For each colour $i\in [k]$ and vertex $u\in V(G)$, we let $B_{i,u}$ be the random bad event that there is no vertex of colour $i$ in $\cstN(u)$, and we let $\Gamma$ be the dependency graph of those bad events.
    Note that $|\cstN(u)|\ge c\cdot \deg(u) \ge c\delta \ge 3$ by assumption. 
    We show with an application of \Cref{cor:LLL2} that, with non-zero probability, no bad event occurs, in which case the realisation of $\bsig$ is a total-dominating $k$-colouring of $G$.
    
    Let $u\in V(G)$ be a fixed vertex, and $i\in [k]$ a fixed colour.
    We have 
    \begin{equation}\label{eq:probability-bound}
        \pr{B_{i,u}}= \pth{1-\frac{1}{k}}^{\card{\cstN(u)}} \le e^{-\frac{c\cdot \deg(u)}{k}} \le \frac{1}{\delta^3}.
    \end{equation}
    Observe that the bad event $B_{i,u}$ is entirely determined by the random choices of colours of the vertices in $\cstN(u)$. For every $u' \notin \bigcup_{v\in \cstN(u)} N_G(v)$, the sets $\cstN(u)$ and $\cstN(u')$ are disjoint, so for every colour $j \in [k]$, the bad events $B_{i,u}$ and $B_{j,u'}$ are independent. 
    Therefore, in the dependency graph $\Gamma$, we have 
    \[
        N_\Gamma(B_{i,u}) \subseteq \bigcup_{j\in [k], v\in \cstN(u)} \sst{B_{j,w}}{w\in N_G(v)}.
    \]
    Although the right-hand side may contain $B_{i,u}$ itself, this only enlarges the set over which we sum.
    By definition, for each $v\in \cstN(u)$, $N_G(v)\setminus N_G(u)$ can be decomposed into $K$ parts $X_1(v),\ldots,X_K(v)$ such that $\deg_G(w)+1\ge p|X_\iota(v)|$ for every $\iota\in[K]$ and $w\in X_\iota(v)$.

    Let $\alpha \coloneqq 1-\frac{1}{3\log\delta}<1$. By \eqref{eq:probability-bound}, we have
    \[
    \pr{B_{i,u}}^\alpha \le \frac{1}{\delta^{3\alpha}} = \frac{e}{\delta^3}<\frac12
    \]
    and
    \begin{align*}
        \sum_{B_{j,w} \in N_\Gamma(B_{i,u})} \pr{B_{j,w}}^\alpha &\le \sum_{j\in [k]} \pth{\sum_{w\in N(u)} \pr{B_{j,w}}^\alpha + \sum_{v\in \cstN(u)} \sum_{\iota \in [K]} \sum_{w\in X_\iota(v)} \pr{B_{j,w}}^\alpha} \\
        &\le k  \pth{\frac{\deg(u)}{\delta^{3\alpha}} + \sum_{v\in \cstN(u)} \sum_{\iota \in [K]} \sum_{w\in X_\iota(v)} e^{-\frac{\alpha c\cdot \deg(w)}{k}}}.
    \end{align*}
    Let $j\in [k]$, $v\in \cstN(u)$, and $\iota\in [K]$ be fixed. Let us bound the contribution of $X_\iota(v)$ in the above sum. We write $t\coloneqq |X_\iota(v)|$; either $pt\le \delta+1$, and 
        \[
        \sum_{w \in X_\iota(v)} e^{-\frac{\alpha c\cdot \deg(w)}{k}}\le \frac{\delta+1}{p} e^{-\frac{\alpha c\delta}{k}} \le \frac{\delta+1}{p\delta^{3\alpha}} < \frac{3}{p\delta^2},
        \] 
     
    where we have used that $e\,\frac{\delta+1}{\delta} < 3$ when $\delta \ge 10$;
    otherwise $pt \ge \delta+1$. Since $\deg(w)+1\ge pt$ for every $w\in X_\iota(v)$, we have
    \[
        \deg(w)\ge pt-1 \ge \frac{pt\delta}{\delta+1}.
    \]
    Hence, setting $\beta\coloneqq \frac{\alpha cp\delta}{k(\delta+1)}$, we get
    \[
        \sum_{w \in X_\iota(v)} e^{-\frac{\alpha c\cdot \deg(w)}{k}}
        \le t e^{-\beta t}.
    \]
    The function $x\mapsto x e^{-\beta x}$ is non-increasing on $[1/\beta,\infty)$. Moreover,
    \[
        t\ge \frac{\delta+1}{p}\ge \frac{k(\delta+1)}{\alpha cp\delta}=\frac{1}{\beta},
    \]
    where we used $k\le \alpha c\delta$. Therefore,
    \[
        t e^{-\beta t}
        \le \frac{\delta+1}{p}e^{-\frac{\alpha c\delta}{k}}
        \le \frac{\delta+1}{p\delta^{3\alpha}}
        < \frac{3}{p\delta^2}.
    \]
    
    We conclude that
    \[
    \sum_{B_{j,w} \in N_\Gamma(B_{i,u})} \pr{B_{j,w}}^\alpha 
    \le k \pth{\frac{e\deg(u)}{\delta^3} + \frac{3\card{\cstN(u)}K}{p\delta^2}}.
    \]
    Moreover, since $k\ge 2$, using the standard inequality $-\log(1-x)\ge x$ for $x\in(0,1)$ with $x=1/k$, we obtain
    \[
        \log \frac{1}{\pr{B_{i,u}}}
        = \card{\cstN(u)}\log\frac{k}{k-1}
        \ge \frac{\card{\cstN(u)}}{k}.
    \]
    Therefore,
    \[
        (1-\alpha)\log_4 \frac{1}{\pr{B_{i,u}}}
        \ge \frac{1}{3\log 4 \cdot \log \delta} \cdot \frac{\card{\cstN(u)}}{k}.
    \]
    So \eqref{eq:LLL2} is satisfied as long as
    \[  \frac{k^2\log \delta\log 4}{\delta^2}\pth{\frac{3e}{c\delta}+\frac{9K}{p}} \le 1.\]
    Using the assumption $c\delta\ge 3$, this holds when $\delta \ge 4^{e/9+K/p}$. We may therefore apply \Cref{cor:LLL2} and thus end the proof.
\end{proof}

We now show that star-free graphs satisfy the setting of \Cref{thm:main-technical}. We first rely on the following structural result.

\begin{lemma}
    \label{lem:large-degree-Kr-free}
    Let $r\ge 2$ be an integer, let $G$ be an $n$-vertex $K_r$-free graph, and let $X$ be the set of vertices of degree more than $\frac{r-2}{r-1}\, n$ in $G$. Then $G[X]$ is $K_{r-1}$-free.
\end{lemma}

\begin{proof}
    If $G$ is $K_{r-1}$-free, the result is immediate.
    Otherwise, let $W$ be any copy of $K_{r-1}$ in $G$. We have $\bigcap_{w \in W} N(w) = \emptyset$ since $G$ is $K_r$-free. Hence
    \[ \sum_{w\in W} \deg(w) \le (r-2)n.\]
    By the Pigeonhole Principle, we infer that $W \not\subseteq X$, and since $W$ was chosen arbitrarily, we conclude that $G[X]$ is $K_{r-1}$-free, as desired.
\end{proof}

As a consequence, we obtain the following.

\begin{cor}\label{cor:constrained-star-free}
	Given an integer $r\ge 2$, let $p \coloneqq \frac{1}{r}$ and $K \coloneqq r$. Then, for every induced-$K_{1,r+1}$-free graph $G$, every vertex $v\in V(G)$ verifies $\cstN(v) = N(v)$.
\end{cor}

\begin{proof}

    Let $v\in V(G)$, and let $u\in N(v)$. Since $G$ is induced-$K_{1,r+1}$-free, we have $\alpha(G[N(u)])\le r$. Hence, for the induced subgraph 
    \[
        H \coloneqq G[N(u)\setminus N(v)],
    \]
    we also have $\alpha(H)\le r$.
    
    We shall use the following claim.
    \begin{claim}
        Let $H$ be an $n$-vertex graph with bounded independence number $\alpha(H)\le r$. 
        Then $V(H)$ can be decomposed into $X_1, \ldots, X_r$ such that, for every $i\in [r]$ and $v\in X_i$, 
        \[ \deg_H(v) \ge \frac{|X_i|}{r}-1.\]
    \end{claim}

    \begin{proofclaim}
        Let $n\coloneqq |V(H)|$.
        If $\delta(H) \ge \frac{n}{r}-1$, the result trivially holds by taking $X_1 = V(H)$ and $X_2 = \ldots = X_r = \emptyset$. So we assume that there is a vertex $x_0$ with $\deg_H(x_0)<\frac{n}{r}-1$.
        Let $X \coloneqq \sst{x\in V(H)}{\deg_H(x) < \frac{n}{r}-1}$; in particular $x_0 \in X$ so $X\neq \emptyset$.
        Since $\alpha(H)\le r$, the complement $\overline{H}$ is $K_{r+1}$-free. Moreover, by the definition of $X$, a vertex $x\in X$ has degree more than $\frac{r-1}{r}n$ in $\overline{H}$. Therefore, by \Cref{lem:large-degree-Kr-free} applied to $\overline{H}$ with parameter $r+1$, the graph $\overline{H}[X]$ is $K_r$-free, i.e. $\alpha(H[X])\le r-1$.
        
        We now proceed by induction on $r$.
        In the base case $r=2$, we have $\alpha(H[X])\le 1$, so $X$ is a clique. Moreover, for every $v\in V(H)\setminus X$, by the definition of $X$, we have
        \[
            \deg_H(v)\ge \frac{n}{2}-1 \ge \frac{|V(H)\setminus X|}{2}-1.
        \]
        Thus $X,V(H)\setminus X$ is the desired decomposition.
        
        In the general case, by the induction hypothesis applied to $H[X]$, one can decompose $X$ into $X_1, \ldots, X_{r-1}$ such that, for every $i\in [r-1]$ and $v\in X_i$,
        \[
            \deg_{H[X]}(v) \ge \frac{|X_i|}{r-1}-1.
        \]
        In particular,
        \[
            \deg_H(v) \ge \deg_{H[X]}(v) \ge \frac{|X_i|}{r-1}-1 \ge \frac{|X_i|}{r}-1.
        \]
        Let $X_r \coloneqq V(H) \setminus X$. For every $v\in X_r$, by the definition of $X$, we have
        \[
            \deg_H(v)\ge \frac{n}{r}-1 \ge \frac{|X_r|}{r}-1.
        \]
        Thus $X_1,\ldots,X_r$ is the desired decomposition of $V(H)$.
    \end{proofclaim}
    
    Applying the claim to $H=G[N(u)\setminus N(v)]$, we obtain a decomposition
    \[
        N(u)\setminus N(v)=X_1\cup\cdots\cup X_r
    \]
    such that, for every $i\in[r]$ and every $w\in X_i$,
    \[
        \deg_H(w)\ge \frac{|X_i|}{r}-1.
    \]
    Since $\deg_G(w)\ge \deg_H(w)$, this gives
    \[
        \deg_G(w)+1\ge \frac{|X_i|}{r}.
    \]
    Thus $N(u)\setminus N(v)$ admits the required decomposition with $p=1/r$ and $K=r$. Therefore $u\in\cstN(v)$. Since $u$ was an arbitrary neighbour of $v$, we conclude that $\cstN(v)=N(v)$.
\end{proof}

As an immediate consequence of \Cref{cor:constrained-star-free} and \Cref{thm:main-technical}, we conclude that for every fixed $r\ge 2$, the class of induced-$K_{1,r+1}$-free graphs is DOM-bounded with a quasilinear DOM-binding function.

\begin{cor}
	Fix $r \ge 2$. Let $G$ be an induced-$K_{1,r+1}$-free graph of minimum degree $\delta \ge 4^{r^2+e/9}$. Then
	\[\tdom(G) \ge \floor{\frac{\delta}{3 \log \delta}}\]
\end{cor}

Since $\tdom(G)\le \dom(G)$, this immediately implies \Cref{thm:main}.

\section{Open problems}
\label{sec:conclusion}

\subsection{Unit disk graphs}

In \Cref{sec:stars}, we first considered the class of unit disk graphs to illustrate a simple application of the Lovász Local Lemma to yield a quasilinear lower bound on the domatic number depending only on the minimum degree, before we moved on to the more abstract setting of $(p,K)$-constrained neighbourhoods we needed to prove DOM-boundedness for the more general classes of induced-star-free graphs. While we proved a quasilinear upper bound for the domatic number of induced-star-free graphs (cf. \Cref{prop:alpha2}), we have no reason to believe that our quasilinear lower bound for unit disk graphs (cf. \Cref{thm:unit-disk}) has the right order of magnitude.

\begin{problem}
	Is there a linear DOM-binding function for the class of unit disk graphs? 
\end{problem}

More generally, we believe DOM-boundedness may be interesting to study in other geometric intersection graphs, such as intersection graphs of balls or boxes in $\mathbb{R}^d$, $d$-directional segment graphs, string graphs, ...

\subsection{Beyond cographs}

Another characterisation of cographs corresponds to the class of graphs of cliquewidth at most $2$. One could consider dominating colourings of graphs of bounded cliquewidth to try and extend the present results. 

\begin{problem}
	Given a fixed integer $t\ge 3$, is the class of graphs of clique-width at most $t$ DOM-bounded? 
	The special case of distance-hereditary graphs (that have clique-width $\le 3$) can be considered first.
\end{problem}

\subsection{Forbidden induced subgraphs}

As a partial progress towards solving \Cref{conj:main}, one could try and solve its restriction to split graphs.
Let $\cS$ denote the class of split graphs, and $S_{s,t}$ the double star with $s+t$ leaves.
\begin{problem}
	Is $\forb(S_{r,r})\cap \cS$ DOM-bounded?
\end{problem}

\subsection{Line graphs}

We showed in \cref{thm:linegraph-lower} that if $G$ is the line graph of a simple graph, then $\dom(G) \ge \frac{\delta}{2} -1$. We then remarked that this lower bound was nearly tight, as there exist line graphs $G$ such that $\dom(G) \le \frac{\delta}{2} + 1$.

We note that the notion of \emph{edge-domatic partitions} (i.e. domatic partition in the line graph) was	considered by Zelinka \cite{zelinka1983edge}. In particular, it was claimed that every graph $G$ has edge-domatic number at least $\delta(G)$; since $L(G)$ has minimum degree at most $\delta(L(G)) \le 2\delta(G) - 2$, this would imply the tight lower bound $\dom(L(G)) \ge \delta(G) \ge \frac{\delta(L(G))}{2} + 1$. However, the proof provided in \cite[Theorem 1]{zelinka1983edge} is based on a flawed argument; roughly speaking, the author repeatedly extracts a maximal independent set of edges, ensuring that the minimum degree decreases by at most $1$ at each step, but fails to ensure that edges extracted at a given step are dominated by every subsequent colour class. We therefore pose Zelinka's claimed bound as a problem.

\begin{problem}
	Is it true that every graph $G$ verifies $\dom(L(G)) \ge \delta(G)$?
\end{problem}

\bibliography{dom-bounded}
\bibliographystyle{plain}
\end{document}